\documentclass[a4paper,11pt]{amsart}
\usepackage[T1]{fontenc}
\usepackage{mathrsfs}
\usepackage{graphicx}
\usepackage{enumerate}
\usepackage{multicol}
\usepackage{color}
\usepackage{amsmath,amssymb}
\usepackage{amsthm}
\usepackage{textcomp}
\usepackage{mathabx}

\usepackage{times}
\usepackage[varg]{txfonts}

\usepackage[top=1.5in, bottom=1.5in, left=3cm, right=3cm]{geometry}

\theoremstyle{plain}
\newtheorem{thm}{Theorem}
\newtheorem{lem}{Lemma}
\newtheorem{cor}{Corollary}

\theoremstyle{definition}
\newtheorem{dfn}{Definition}
\newtheorem{ex}{Example}

\theoremstyle{remark}
\newtheorem*{rem}{Remark}

\newtheorem*{ackn}{Acknowledgment}

\title[A characterization of differential operators]{A characterization of differential operators in the ring of complex polynomials}
\author{Włodzimierz Fechner and Eszter Gselmann}



\begin{document}

\begin{abstract}
The paper aims to provide a full characterization of all operators $T\colon \mathscr{P}(\mathbb{C}) \to \mathscr{P}(\mathbb{C})$ acting on the space of all complex polynomials that satisfy the Leibniz rule
\[
T(f\cdot g)= T(f)\cdot g+f\cdot T(g)
\]
for all $f, g\in \mathscr{P}(\mathbb{C})$. We do not assume the linearity of $T$.
As we will see, contrary to the well-known theorems for function spaces there are many other solutions here, not only differential operators.
From our main result, we also derive two corollaries, showing that in some special cases operators that satisfy the Leibniz rule have some particular form.
\end{abstract}

\maketitle

\section{Introduction}

Characterizations of basic operators, such as the first-order differential operator or the Fourier transform by one of its fundamental properties and without assuming linearity, are one of the most interesting problems in modern analysis. Let us briefly recall some most important from our point of view results in this direction.

For instance, if $\mathscr{S}(\mathbb{R}^{N})$ denotes the Schwartz space of `rapidly' decreasing functions $f\colon \mathbb{R}^{n}\to \mathbb{C}$, then in Artstein-Avidan--Faifman--Milman \cite{ArtFaiMil12}, the authors showed that any \emph{bijective} transformation $T\colon \mathscr{S}(\mathbb{R}^{n})\to \mathscr{S}(\mathbb{R}^{n})$ that also fulfills
\[
T(f\cdot g)= T(f)\ast T(g)
\qquad
\left(f, g\in \mathscr{S}(\mathbb{R}^{N})\right)
\]
is just a slight modification of the Fourier transform.

In addition, we can mention a characterization theorem of higher order differential operators due to Peetre \cite{Pee59}, see also \cite{BraEllRob86, Eng67}.
Let $n$ be a positive integer and let $\mathscr{C}^{\infty}_{0}(\mathbb{R}^{n})$ denote the set of infinitely many times differentiable functions defined on $\mathbb{R}^{n}$ that vanish at infinity, further let $\mathscr{C}_{b}(\mathbb{R}^{n})$ stand for the set of all those continuous functions defined on $\mathbb{R}^{n}$ that are bounded. If the \emph{linear} operator $T\colon \mathscr{C}^{\infty}_{0}(\mathbb{R}^{n})\to \mathscr{C}_{b}(\mathbb{R}^{n})$ fulfills
\[
\mathrm{supp}(Tf)\subset \mathrm{supp}(f)
\]
for all $f\in \mathscr{C}^{\infty}_{0}(\mathbb{R}^{n})$, then there exists a positive integer $N$
and for all $\alpha\in \mathbb{N}^{n}$, $|\alpha|\leq N$, there exists continuous functions $c_{\alpha}$ on $\mathbb{R}^{N}$ such that
\[
(Tf)(x)= \sum_{|\alpha|\leq N}c_{\alpha}(x)D^{\alpha}f(x)
\]
for all $f\in \mathscr{C}^{\infty}_{0}(\mathbb{R}^{n})$ and $x\in \mathbb{R}^{n}$, where
\[
D^{\alpha}f(x)= \frac{\partial^{|\alpha|}}{\partial x_{1}^{\alpha_{1}}\cdots \partial x_{n}^{\alpha_{n}}}f(x_{1}, \ldots, x_{n}),
\]
and $x=(x_{1}, \ldots, x_{n})\in \mathbb{R}^{n}$.

The main objective of our paper is to prove similar characterization theorems in polynomial rings. More precisely, we will study here those operators $T\colon \mathscr{P}(\mathbb{K})\to \mathscr{P}(\mathbb{K})$ that fulfill the Leibniz rule, i.e.
\[
T(f\cdot g)= f T(g)+T(f)g
\qquad
\left(f, g\in \mathscr{P}(\mathbb{K})\right).
\]
We would like to specifically emphasize that in this work we do not impose any additional conditions (such as linearity or bijectivity) on the operator $T$.

As we will see, unlike the well-known theorems for function spaces (see Theorem \ref{thmKM}), there are many other solutions here, not only differential operators. Moreover, if $D^{\alpha}$ is a differential operator acting on $\mathbb{R}$ or on $\mathbb{C}$, then
\[
\mathrm{deg}\left(D^{\alpha} p\right) \leq \mathrm{deg}(p)
\]
holds for any polynomial $p$.
Thus after presenting our main result, we focus on determining the solutions of the Leibniz rule that also have this property, see Corollaries \ref{cor1} and \ref{cor2}.

\subsection{Leibniz rule in $\mathscr{C}^{k}(I)$}

Let $I\subset \mathbb{R}$ be an open set, $k$ be a nonnegative integer
let
\[
\mathscr{C}^{k}(I)=
\left\{ f\colon I\to \mathbb{R}\, \vert \, \text{$f$ is $k$-times continuously differentiable on $I$}\right\}.
\]
We denote the space of continuous functions also by $\mathscr{C}(I)$, instead of writing $\mathscr{C}^{0}(I)$, further we put
\[
\mathscr{C}^{\infty}(I)= \bigcap_{k} \mathscr{C}^{k}(I).
\]

On spaces $\mathscr{C}^{k}(I)$ the solutions of the Leibniz rule were determined by K\"onig and Milman
\cite{KonMil18}. Let us quote their result.

\begin{thm}[Leibniz rule on $\mathscr{C}^{k}(I)$]\label{thmKM}
Let $I\subset \mathbb{R}$ be an open set and $k$ be a nonnegative integer. Suppose that the operator $T\colon \mathscr{C}^{k}(I)\to \mathscr{C}(I)$ satisfies the Leibniz rule, i.e.,
\begin{equation}\label{L}
T(f\cdot g)= T(f)\cdot g+f\cdot T(g)
\end{equation}
holds for all $f, g\in \mathscr{C}^{k}(I)$. Then there exists continuous functions $c, d\in \mathscr{C}(I)$ such that
\[
T(f)=
\begin{cases}
cf\ln(|f|) +df', & \text{ if } k\geq 1\\
cf\ln(|f|) , & \text{ if } k=0.
\end{cases}
\qquad
\left(f\in \mathscr{C}^{k}(I)\right).
\]
Conversely, any map $T$ defined by the above formula fulfils the Leibniz rule on $\mathscr{C}^{k}(I)$.
\end{thm}

The proof of this theorem consists of two steps. At first, one has to show that the operator $T$ is `localized on intervals'. In other words, the operator $T$ is defined pointwise in that sense that there is a function $F\colon I\times \mathbb{R}^{n+1}\to \mathbb{R}$ such that for all $f\in \mathscr{C}^{k}(I)$ and $x\in I$ we have
\[
T(f(x))= F(x, f(x), \ldots, f^{(k)}(x)).
\]
At this point, no regularity of $F$ is known. Observe that then the Leibniz rule is equivalent to a functional equation for the representing function $F$. The second step is to analyze the structure of $F$ and to prove the continuity of the coefficient functions occurring there, by using the fact that the image space of the operator $T$ is $\mathscr{C}(I)$.

\subsection{Derivations}

The Leibniz rule, that is, equation \eqref{L} plays a key role not only in operator theory but also in algebra. Here we recall some basic facts from the monograph Kuczma \cite{Kuc09}.

\begin{dfn}
Let $(Q;+,\cdot)$ be a commutative ring, and let $(P;+,\cdot)$ be a subring
of $(Q;+,\cdot)$. A function $f:P\to Q$ is called a \emph{derivation} if it satisfies both the equations
\[
f(x+y)=f(x)+f(y),
\]
and also
\[
f(xy)=xf(y)+yf(x)
\]
for all $x,y\in P$.

If $f$ satisfies the second equation, then we will call it the Leibniz mapping.
\end{dfn}

\begin{ex}\label{Ex14.1.1}
Let $(F;+,\cdot)$ be a field, and let $P=Q=F\left[x\right]$ be the ring
of polynomials with coefficients from $F$. Let the function
$f:F\left[x\right]\to F\left[x\right]$ be defined as
$f(p)=p'$, where $p'$ is the derivative of $p$.
We have clearly
\[
f(p+q)=(p+q)'=p'+q'=f(p)+f(q)\,,
\]
\[
f(pq)=(pq)'=pq'+qp'=pf(q)+qf(p)\,.
\]
Consequently, $f$ is a derivation.
\end{ex}
\begin{ex}
Let $(F;+,\cdot)$ be a field, and suppose that we are given a derivation
$f:F\to F$. We define a function
$f_0:F\left[x\right]\to F\left[x\right]$ as follows. If $p\in F\left[x\right]$,
$p(x)=\sum\limits^{n}_{k=0}a_k x^{k}$, consider the polynomial $f_0(p)$ defined by
(in sequel denoted by $p^{f}$)
\[
f_0(p)=p^{f}(x)=\sum\limits^{n}_{k=0}f(a_k)x^{k}.
\]
Then $f_{0}\colon F[x]\to F[x]$ is a derivation.
\end{ex}
The derivations described in the above two examples have a fundamental importance
for we have the following lemma, which is a modification of Kuczma \cite[Lemma 14.2.2]{Kuc09}.

\begin{lem}
Let $(K;+,\cdot)$ be a field, and let $(F;+,\cdot)$ be a subfield of
$(K;+,\cdot)$, and let $f:F\to K$ be a derivation. Then we have, for every
$a\in F$ and every polynomial $p\in F\left[x\right]$,
\[
f\big(p(a)\big)=p^{f}(a)+f(a)p'(a).
\]
\end{lem}
%
%

The question of whether there exists a nontrivial, i.e. nonzero derivation on fields with characteristic zero is far from being obvious but can be answered affirmatively. The following theorem (see Kuczma \cite[Lemma 14.2.4]{Kuc09}) plays a key role in addressing this problem.

\begin{thm}
Let $(K;+,\cdot)$ be a field of characteristic zero, let $(F;+,\cdot)$
be a subfield of $(K;+,\cdot)$, let $S$ be an algebraic base of $K$ over $F$,
if it exists, and let $S=\varnothing$ otherwise.
Let $f:F\to K$ be a derivation.
Then, for every function $u:S\to K$,
there exists a unique derivation $g:K\to K$
such that $g \mid F=f$ and $g \mid S=u$.
\end{thm}

Let now $\mathbb{K}\in \left\{\mathbb{R}, \mathbb{C}\right\}$. Since $\mathrm{algcl}(\mathbb{Q})\neq \mathbb{K}$, there exists an algebraic base of $\mathbb{K}$ over $\mathbb{Q}$. Therefore we have the following.

\begin{thm}
Let $\mathbb{K}\in \left\{\mathbb{R}, \mathbb{C}\right\}$. There exists a non-identically zero derivation $d\colon \mathbb{K}\to \mathbb{K}$.
\end{thm}

\section{Differential operators in polynomial rings}

Let $\mathbb{K}\in \left\{ \mathbb{R}, \mathbb{C}\right\}$ and consider
\[
\mathscr{P}(\mathbb{K}) =
\left\{ p\colon \mathbb{K}\to \mathbb{K}\, \vert \, \text{$p$ is a polynomial} \right\}.
\]

Below we present some examples that show that on $\mathscr{P}(\mathbb{K})$ there exists a large variety of mappings that satisfy the Leibniz rule. Contrary to the case of the spaces of smooth functions that was discussed by K\"onig and Milman, see Theorem \ref{thmKM} below.

\subsection{Examples}

\begin{dfn}
Let $p\in \mathscr{P}(\mathbb{K})$ be a polynomial, $x_{0}\in \mathbb{K}$ and $k$ be a nonnegative integer. We say that the \emph{order of zero} of the polynomial $p$ at $x_{0}$ is $k$, if there exists a polynomial $q\in \mathscr{P}(\mathbb{K})$ such that $q(x_{0})\neq 0$ and
\[
p(x)= (x-x_{0})^{k}\cdot q(x)
\qquad
\left(x\in \mathbb{K}\right).
\]
The order of zero of a polynomial $p$ at the point $x_{0}$ will be denoted by $n_{x_{0}}(p)$.
\end{dfn}

\begin{ex}
Let $x_{0}\in \mathbb{K}$ be arbitrarily fixed and define the mapping $N\colon \mathscr{P}(\mathbb{K}) \to \mathscr{P}(\mathbb{K})$ by
\[
N(p)= n_{x_{0}}(p) \cdot p
\qquad
\left(p\in \mathscr{P}(\mathbb{K})\right).
\]
Then $N\colon \mathscr{P}(\mathbb{K})\to \mathscr{P}(\mathbb{K})$ fulfills the Leibniz rule on $\mathscr{P}(\mathbb{K})$.
\end{ex}

\begin{ex}
The mapping $F\colon \mathscr{P}(\mathbb{K}) \to \mathscr{P}(\mathbb{K})$ defined through
\[
F(p)= \mathrm{deg}(p) \cdot p
\qquad
\left(p\in \mathscr{P}(\mathbb{K})\right),
\]
where $\mathrm{deg}(p)$ denotes the degree of the polynomial $p$, fulfills the Leibniz rule on $\mathscr{P}(\mathbb{K})$.
\end{ex}

\begin{ex}
Let us fix a polynomial $p_0\in \mathscr{P}(\mathbb{K})$ and consider the mapping $P\colon \mathscr{P}(\mathbb{K}) \to \mathscr{P}(\mathbb{K})$ defined by
\[
P(p)=p' \cdot p_0
\qquad
(p\in \mathscr{P}(\mathbb{K})).
\]
Then $P$ satisfies the Leibniz rule on $\mathscr{P}(\mathbb{K})$. From this example, it follows that an action of an operator satisfying the Leibniz rule on a polynomial can increase its degree.
\end{ex}

\begin{ex}
The mapping $E\colon \mathscr{P}(\mathbb{K}) \to \mathscr{C}(\mathbb{K})$ defined through
\[
E(p)= p \cdot \ln(|p| )
\qquad
\left(p\in \mathscr{P}(\mathbb{K})\right)
\]
satisfies the Leibniz rule. Here we adopt the convention $0\cdot \ln(0)=0$.
\end{ex}

\begin{ex}
Let $d\colon \mathbb{K}\to \mathbb{K}$ be a derivation, then the mapping
\[
K(p)= p^{d}
\]
also fulfills the Leibniz rule on $\mathscr{P}(\mathbb{K})$.
\end{ex}

\begin{ex}
Let $f\colon \mathbb{C}\to \mathbb{N}$ be an arbitrary function and let $q_0\in \mathscr{P}(\mathbb{C})$ be a fixed polynomial.
Define map $Q\colon \mathscr{P}(\mathbb{C}) \to \mathscr{P}(\mathbb{C})$ as follows. If $p$ is a constant polynomial, then put $Q(p)=0$. If $p$ is nonconstant, then it decomposes uniquely into linear terms, i.e.,
$$
p(z)= a\cdot \prod_{j=1}^{N}(z-z_{j}).
$$
In this case, define
\[
Q(p) = a\cdot \sum_{k=1}^N q_0^{f(z_k)}\cdot \prod_{j=1, j \neq k}^{N}(z-z_{j}).
\]
We show that $Q$ satisfies the Leibniz rule. Fix $p, q \in \mathscr{P}(\mathbb{C})$. If $p$ or $q$ is constant, then the assertion is true by the definition of $Q$ (note that $Q$ is a homogeneous mapping). Next, assume that $p$ and $q$ are nonconstant. Therefore, they decompose as
$$
p(z)= a_1\cdot \prod_{j=1}^{N}(z-z_{j}), \qquad q(z)= a_2\cdot \prod_{j=N+1}^{M}(z-z_{j}).
$$
Therefore
\begin{align*}
Q(p)\cdot q+p\cdot Q(q) &=
a_1\cdot \left(\sum_{k=1}^n q_0^{f(z_k)}\cdot \prod_{j=1, j \neq k}^{N}(z-z_{j})\right) \cdot a_2\cdot \prod_{j=N+1}^{M}(z-z_{j}) \\&+ a_2\cdot \left(\sum_{k=N+1}^M q_0^{f(z_k)}\cdot \prod_{j=N+1, j \neq k}^{M}(z-z_{j})\right) \cdot a_1\cdot \prod_{j=1}^{N}(z-z_{j}) \\&= a_1a_2\cdot \left(\sum_{k=1}^N q_0^{f(z_k)}\cdot \prod_{j=1, j \neq k}^{M}(z-z_{j})\right)
\\&+ a_1a_2\cdot \left(\sum_{k=N+1}^M q_0^{f(z_k)}\cdot \prod_{j=1, j \neq k}^{M}(z-z_{j})\right) = Q(p \cdot q).
\end{align*}
Since map $f$ can be arbitrary, then the degree of a polynomial $Q(p)$ can be arbitrarily large, even if $p$ is a polynomial of degree $1$.
\end{ex}

\begin{rem}
An analogous example can be given for real polynomials, as above.
Indeed, if $p\in \mathscr{P}(\mathbb{R})$, then
\[
p(z)= a \prod_{i=1}^{n}(z-z_{i})\cdot \prod_{j=1}^{l}(z^{2}+\alpha_{j}z+\beta_{j})
\qquad
\left(z\in \mathbb{R}\right),
\]
where for all $j=1, \ldots, l$, the polynomials $z^{2}+\alpha_{j}z+\beta_{j}$ are irreducible over $\mathbb{R}$.

If $n=0$ and $k=0$, so if $p$ is a constant, then we simply put $Q(p)=0$. Otherwise, let
\[
Q(p)= a \left(\sum_{k=1}^n q_0^{f(z_k)}\cdot \prod_{j=1, j \neq k}^{n}(z-z_{j})\right)\cdot \prod_{j=1}^{l}(z^{2}+\alpha_{j}z+\beta_{j}).
\]
A similar computation as was carried out above shows that the mapping $Q\colon \mathscr{P}(\mathbb{R})\to \mathscr{P}(\mathbb{R})$ fulfils the Leibniz rule. Again, since $f$ can be any mapping, the polynomial $Q(p)$ can admit arbitrarily large degree.
\end{rem}

\begin{rem}
Observe that the Leibniz rule \eqref{L}
is a homogeneous and linear equation for the unknown operator $T\colon \mathscr{P}(\mathbb{K})\to \mathscr{C}(\mathbb{K})$. Thus its solution space forms a linear space. Therefore, any linear combination of the above operators $N, F, P, E, K$ and $Q$, also satisfies the same identity.
\end{rem}

\subsection{Complex polynomials}

In this subsection, we present our main result, which is a characterization of all operators on $\mathscr{P}(\mathbb{C})$ that satisfy the Leibniz rule. Thus here we consider the case $\mathbb{K}= \mathbb{C}$.

Recall, that due to the Theorem of Algebra, if $p\in \mathscr{P}(\mathbb{C})$, then we have
\[
p(z)= a\cdot \prod_{j=1}^{N}(z-z_{j})
\]
with an appropriate nonnegative integer $N$ and with some complex numbers $a, z_{1}, \ldots, z_{N}$.

Let now $T\colon \mathscr{P}(\mathbb{C}) \to \mathscr{C}(\mathbb{C})$ be an operator that fulfills the Leibniz rule. Then, by induction, we obtain that
\[
T(p_{1} \cdots p_{n})= \sum_{j=1}^{n}\left(\prod_{i\neq j}p_{i}\right) \cdot T(p_{j})
\qquad
\left(p_{1}, \ldots, p_{n}\in \mathscr{P}(\mathbb{C})\right)
\]
for all positive integer $n$ and for all $p_{1}, \ldots, p_{n}\in \mathscr{P}(\mathbb{C})$.

Thus we immediately get that
\[
T(p)(z)= a\cdot \sum_{j=1}^{N}\left(\prod_{i\neq j} (z-z_{i})\right) \cdot T(z-z_{j})+
T(a)\prod_{j=1}^{N}(z-z_{j}).
\]
holds for all $p\in \mathscr{P}(\mathbb{C})$. From this, we obtain that it is enough to determine the action of the operator $T$ on complex polynomials that are of the form
\[
p(z)=az+b
\qquad
\left(z\in \mathbb{C}\right)
\]
with some appropriate complex constants $a, b$.

In the case when the domain of operator $T$ is equal to the space $\mathscr{C}^k(I)$, then a localization lemma \cite[Lemma 3.2]{KonMil18} plays a key role. It says that the value $T(p)(z)$ of operator $T$ acting on a polynomial $p$ at a point $z$ depends only on $z, p(z), p'(z), \dots , p^{(k)}(z)$. In the case of the space $\mathscr{P}(\mathbb{C})$ the localization property is not so powerful. Thus while proving our main result, which is stated below, we need to apply a new approach. We begin with the formulation of a localization result for the space $\mathscr{P}(\mathbb{C})$, which is a consequence of our Theorem \ref{t1} below.

\begin{lem}[Localization]
Suppose that the mapping $T\colon \mathscr{P}(\mathbb{K})\to \mathscr{C}(\mathbb{K})$ fulfills the Leibniz rule \eqref{L} for all $p, q\in \mathscr{P}(\mathbb{K})$.
Then $T$ is \emph{localized}, i.e., there exists a function $F\colon \mathbb{K}\times \mathbb{K}\to \mathbb{K}$ such that
\[
T(p)(z)= F(z, p(z))
\qquad
\left(z\in \mathbb{K}\right).
\]
Moreover, if the range of the operator $T$ is $\mathscr{P}(\mathbb{K})$, then $F$ is a two-variable $\mathbb{K}$-valued \emph{polynomial}.
\end{lem}

Now, we will state and prove our main result without assuming the localization property.

\begin{thm}\label{t1}
The mapping $T\colon \mathscr{P}(\mathbb{C})\to \mathscr{P}(\mathbb{C})$ fulfills the Leibniz rule \eqref{L}
for all $p, q\in \mathscr{P}(\mathbb{C})$ if and only if there exist sequences of functions $(\psi_{k})_{k\in \mathbb{N}_{0}}$ and $(\tilde{\varphi_{k}})_{k\in \mathbb{N}_{0}}$ such that for all $k\in \mathbb{N}_{0}$ we have
$\tilde{\varphi}_{k}(ab)= a\tilde{\varphi}_{k}(b)+b\tilde{\varphi}_{k}(a)$ for all $a, b\in \mathbb{C}$ such that
\[
T(p)(z)=
a\cdot \sum_{j=1}^{N}\left(\prod_{i\neq j} (z-z_{i})\right) \cdot \left(\sum_{k=0}^{n}\psi_{k}\left(-z_{j}\right) z^{k} +\sum_{k=0}^{n}\tilde{\varphi}_{k}(-z_{j})z^{k} \right)
+\left(\sum_{k=0}^{n} \tilde{\varphi}_{k}(a)z^{k}\right)\prod_{j=1}^{N}(z-z_{j}),
\]
provided that $p(z)= a\prod_{j=1}^{N}(z-z_{j})$ and $n\in \mathbb{N}_0$ depends upon $p$.
\end{thm}
\begin{proof}
Firstly, we consider $\mathscr{P}(\mathbb{C})$-valued solutions of the Leibniz rule \eqref{L}.
Then
\begin{equation}\label{az+b}
T(az+b)= \sum_{k=0}^{n}\varphi_{k}(a, b)z^{k}
\qquad
\left(a, b, z\in \mathbb{C}\right),
\end{equation}
with an appropriate $n\in \mathbb{N}$ that depends upon $a$ and $b$, and functions $\varphi_{k}\colon \mathbb{C}^{2}\to \mathbb{C}$, $k=0, 1, \ldots, n$.

First, let us consider constant polynomials. We have
$$
T(b)(z) = \sum_{k=0}^n \varphi_k(0,b)z^k
$$
for some $n \in \mathbb{N}$ and each $b \in \mathbb{C}$ and $z \in \mathbb{C}$. Due to \eqref{L} we have that for all $b, c\in \mathbb{C}$ there exists some $n \in \mathbb{N}$ that is equal to the greatest of the degree of the polynomials $T(b), T(c), T(bc)$ and
\begin{align*}
T(bc)(z) &= bT(c)(z) + c T(b)(z) \\&= \left( \sum_{k=0}^n \varphi_k(0,b)z^k\right)c + \left( \sum_{k=0}^n \varphi_k(0,c)z^k\right)b
\\&= \sum_{k=0}^n \left(\varphi_k(0,b)c + \varphi_k(0,c)b \right)z^k
\end{align*}
On the other hand,
$$ T(bc)(z) = \sum_{k=0}^{n}\varphi_{k}(0, bc)z^{k} .$$
Comparing the coefficients we arrive at
$$\varphi_k(0,bc) = \varphi_k(0,b)c + \varphi_k(0,c)b$$
for all $k \in \mathbb{N}$ and $b, c \in \mathbb{C}$.

Next, substitute in \eqref{az+b} $b=0$. We have
$$ T(az)= \sum_{k=0}^{n}\varphi_{k}(a, 0)z^{k}.$$
On the other hand, using \eqref{L} first, we get
$$
T(az) = aT(z) + z T(a) = a \sum_{k=0}^{n}\varphi_{k}(1, 0)z^{k} + \sum_{k=0}^{n}\varphi_{k}(0,a)z^{k+1}
$$
for all $a \in \mathbb{C}$ and $z \in \mathbb{C}$. Comparing both polynomials we obtain
$$\varphi_0(a,0) = \varphi_0(1,0)a,$$
$$\varphi_k(a,0) = \varphi_k(1,0)a + \varphi_{k-1}(0,a)$$
for all $k \in \mathbb{N}$ and $a \in \mathbb{C}$.

Let now $a, b \in \mathbb{C}$ be arbitrary scalars such that $a\neq 0$. By \eqref{L} and the previous observations we have
\begin{align*}
T(az+b) &= T\left( a \cdot \left( z + \frac{b}{a} \right) \right)
= aT\left( z + \frac{b}{a} \right) + \left( z + \frac{b}{a} \right) T(a) \\
&= \sum_{k=0}^{n}a \varphi_{k}\left(1, \frac{b}{a}\right)z^{k} + \sum_{k=0}^{n}\varphi_{k}\left(0, a\right)z^{k+1}
+ \sum_{k=0}^{n}\frac{b}{a}\varphi_{k}\left(0, a\right)z^{k} .
\end{align*}
Compare this equality with \eqref{az+b} to derive the recurrences:
$$\varphi_0(a,b) = \varphi_0\left(1,\frac{b}{a}\right)a + \varphi_0\left(0,a\right)\frac{b}{a},$$
$$\varphi_k(a,b) = \varphi_k\left(1,\frac{b}{a}\right)a + \varphi_{k-1}(0,a) + \varphi_k\left(0,a\right)\frac{b}{a},$$

For all $k=0, 1, \ldots, n$, the mapping
\[
a\longmapsto \varphi_{k}(0, a)
\]
fulfills the Leibniz rule \eqref{L}, we have that
\[
\varphi_{k}(0, b)=\varphi_{k}\left(0, a \frac{b}{a}\right)=
a\varphi_{k}\left(0, \frac{b}{a}\right) + \frac{b}{a}\varphi_{k}\left(0, a \right),
\]
that is,
\[
\frac{b}{a}\varphi_{k}\left(0, a \right) = \varphi_{k}(0, b) - a\varphi_{k}\left(0, \frac{b}{a}\right)
\]
holds for all $a, b\in \mathbb{C}$, $a\neq 0$.

Therefore
\begin{align*}
T(az+b)& =
\sum_{k=0}^{n}\varphi_{k}(a, b)z^{k}=
\varphi_{0}(a, b) +\sum_{k=1}^{n}\varphi_{k}(a, b)z^{k}
\\
&=
\varphi_0\left(1,\frac{b}{a}\right)a + \varphi_0\left(0,a\right)\frac{b}{a}
+
\sum_{k=1}^{n}\left(\varphi_k\left(1,\frac{b}{a}\right)a + \varphi_{k-1}(0,a) + \varphi_k\left(0,a\right)\frac{b}{a}\right) z^{k}
\\
&=\varphi_0\left(1,\frac{b}{a}\right)a+ \varphi_{0}(0, b) - a\varphi_{0}\left(0, \frac{b}{a}\right)
\\
&+\sum_{k=1}^{n} \left(\varphi_k\left(1,\frac{b}{a}\right)a
+ \varphi_{k-1}(0,a) +\varphi_{k}(0, b) - a\varphi_{k}\left(0, \frac{b}{a} \right)\right)z^{k}
\\
&
=\psi_{0}\left(\frac{b}{a}\right)a+\varphi_{0}(0, b)
+
\sum_{k=1}^{n}\left(\psi_{k}\left(\frac{b}{a}\right)a +\varphi_{k-1}(0, a)+\varphi_{k}(0, b)\right)z^{k}
\\
&= \sum_{k=0}^{n} a\psi_{k}\left(\frac{b}{a}\right) z^{k} +\sum_{k=0}^{n}\varphi_{k}(0, b)z^{k} +\sum_{k=1}^{n}\varphi_{k-1}(0, a)z^{k}
\end{align*}
for all $a, b\in \mathbb{C}$ with $a\neq 0$, where for all $k=0, 1, \ldots, n$, the functions $\psi_{k}\colon \mathbb{C}\to \mathbb{C}$ are defined through
\[
\psi_{k}(c)= \varphi_{k}(1, c)-\varphi_{k}(0, c)
\qquad
\left(c\in \mathbb{C}\right).
\]

Define the functions $\tilde{\varphi_{k}}\colon \mathbb{C}\to \mathbb{C}$ through
\[
\tilde{\varphi_{k}}(a)= \varphi_{k}(0, a)
\qquad
\left(x\in \mathbb{C}\right).
\]
Observe that for all $k=0, \ldots, n$, the mapping $\tilde{\varphi_{k}}$ satisfies the Leibniz rule \eqref{L}. So especially $\tilde{\varphi_{k}}(1)=0$.
With these functions, we can deduce that
\[
T(az+b)= \sum_{k=0}^{n} a\psi_{k}\left(\frac{b}{a}\right) z^{k} +\sum_{k=0}^{n}\tilde{\varphi}_{k}(b)z^{k} +\sum_{k=1}^{n}\tilde{\varphi_{k-1}}(a)z^{k}
\]
for all $a, b\in \mathbb{C}$ and $z\in \mathbb{C}$.

Therefore for any fixed $z_{j}\in \mathbb{C}$, we have
\begin{align*}
T(z-z_{j})&=
\sum_{k=0}^{n}\psi_{k}\left(-z_{j}\right) z^{k} +\sum_{k=0}^{n}\tilde{\varphi}_{k}(-z_{j})z^{k} +\sum_{k=1}^{n}\tilde{\varphi}_{k-1}(1)z^{k}
\\
&= \sum_{k=0}^{n}\psi_{k}\left(-z_{j}\right) z^{k} +\sum_{k=0}^{n}\tilde{\varphi}_{k}(-z_{j})z^{k}
\end{align*}
and also for any fixed $a\in \mathbb{C}$,
\[
T(a)= \sum_{k=0}^{n} \tilde{\varphi}_{k}(a)z^{k}.
\]

Together with \eqref{L}, this implies that if $p\in \mathscr{P}(\mathbb{\mathbb{C}})$ with
\[
p(z)= a\prod_{j=1}^{N}(z-z_{j}),
\]
then
\begin{align*}
T(p(z))& =
a\cdot \sum_{j=1}^{N}\left(\prod_{i\neq j} (z-z_{i})\right) \cdot T(z-z_{j})+
T(a)\prod_{j=1}^{N}(z-z_{j}) \\
&= a\cdot \sum_{j=1}^{N}\left(\prod_{i\neq j} (z-z_{i})\right) \cdot \left(\sum_{k=0}^{n}\psi_{k}\left(-z_{j}\right) z^{k} +\sum_{k=0}^{n}\tilde{\varphi}_{k}(-z_{j})z^{k}\right)\\
&+\left(\sum_{k=0}^{n} \tilde{\varphi}_{k}(a)z^{k}\right)\prod_{j=1}^{N}(z-z_{j})
\end{align*}

Conversely, let $p, q\in \mathscr{P}(\mathbb{C})$ with
\[
p(z)= a_{1}\prod_{j=1}^{N}(z-z_{j})
\quad
\text{and}
\quad
q(z)= a_{2}\prod_{j=N+1}^{M}(z-z_{j}).
\]

Then
\begin{align*}
T(p)q&+T(q)p= \left(a_{1}\cdot \sum_{j=1}^{N}\left(\prod_{\substack{i=1 \\i\neq j}} ^{N}(z-z_{i})\right) \cdot T(z-z_{j})+
T(a_1)\prod_{j=1}^{N}(z-z_{j})\right) \cdot a_{2}\prod_{j=N+1}^{M}(z-z_{j})
\\
&+ \left(a_{2}\cdot \sum_{j=N+1}^{M}\left(\prod_{\substack{i=N+1 \\i\neq j}}^{M} (z-z_{i})\right) \cdot T(z-z_{j})+
T(a_2)\prod_{j=N+1}^{M}(z-z_{j})\right) \cdot a_{1}\prod_{j=1}^{N}(z-z_{j})
\\
&= a_{1}a_{2}\cdot \sum_{j=1}^{M} \left(\prod_{\substack{i=1 \\i\neq j}} ^{M}(z-z_{i})\right) \cdot T(z-z_{j}) +(T(a_{1})a_{2}+T(a_{2})a_{1})\prod_{j=1}^{M}(z-z_{j})
\\
&= a_{1}a_{2}\cdot \sum_{j=1}^{M} \left(\prod_{\substack{i=1 \\i\neq j}} ^{M}(z-z_{i})\right) \cdot T(z-z_{j}) +T(a_{1}a_{2})\prod_{j=1}^{M}(z-z_{j})
\\
& = T(pq).
\end{align*}
\end{proof}

From the above theorem, we can derive a few corollaries for some special types of operators. We say that an operator
$T\colon \mathscr{P}(\mathbb{C})\to \mathscr{P}(\mathbb{C})$ \emph{decreases the degree} if for any polynomial $p$ of degree $N$, the degree of $T(p)$ is smaller than $N$. Moreover, operator $T$ does not increase the degree if, for any polynomial $p$ of degree $N$, the degree of $T(p)$ is not greater than $N$.

\begin{cor}\label{cor1}
The mapping $T\colon \mathscr{P}(\mathbb{C})\to \mathscr{P}(\mathbb{C})$ fulfills the Leibniz rule \eqref{L}
for all $p, q\in \mathscr{P}(\mathbb{C})$ and decreases the degree if and only if there exist two functions $\psi_{0}$ and $\tilde{\varphi_{0}}$ such that we have
$\tilde{\varphi}_{0}(ab)= a\tilde{\varphi}_{0}(b)+b\tilde{\varphi}_{0}(a)$ for all $a, b\in \mathbb{C}$ such that
\[
T(p)(z)=
a\cdot \sum_{j=1}^{N}\left(\prod_{i\neq j} (z-z_{i})\right) \cdot \left(\psi_{0}\left(-z_{j}\right) +\tilde{\varphi}_{0}(-z_{j})\right).
\]
provided that $p(z)= a\prod_{j=1}^{N}(z-z_{j})$.

Consequently, for each monomial $p(z) = z^N$ there exists a constant $c_p \in \mathbb{C}$ such that
$$
T(p) = c_p\cdot p'.
$$
\end{cor}
\begin{proof}
Since $T$ decreases the degree, then necessarily $T(b) = 0$ for any constant polynomial $b\in \mathscr{P}(\mathbb{C})$. Therefore,
$$\sum_{k=0}^n \varphi_k(0,b)z^k=0.$$
Moreover, since the image of each polynomial of the degree $1$ is a constant map, we have that $n=0$ in the statement of Theorem \ref{t1}.
\end{proof}

\begin{cor}\label{cor2}
The mapping $T\colon \mathscr{P}(\mathbb{C})\to \mathscr{P}(\mathbb{C})$ fulfills the Leibniz rule \eqref{L}
for all $p, q\in \mathscr{P}(\mathbb{C})$ and not increases the degree if and only if there exist four functions $\psi_{0}$, $\psi_{1}$, $\tilde{\varphi_{0}}$ and $\tilde{\varphi_{1}}$ such that we have
$\tilde{\varphi}_{k}(ab)= a\tilde{\varphi}_{k}(b)+b\tilde{\varphi}_{k}(a)$ for all $a, b\in \mathbb{C}$ and $k\in \{0,1\}$ such that
\[
T(p)(z)=
a\cdot \sum_{j=1}^{N}\left(\prod_{i\neq j} (z-z_{i})\right) \cdot \left(\sum_{k=0}^{1}\psi_{k}\left(-z_{j}\right) z^{k} +\sum_{k=0}^{1}\tilde{\varphi}_{k}(-z_{j})z^{k} \right)
+\tilde{\varphi}_{0}(a)\prod_{j=1}^{N}(z-z_{j}),
\]
provided that $p(z)= a\prod_{j=1}^{N}(z-z_{j})$.

Consequently, for each monomial $p(z) = z^N$ there exist constants $c_p, d_p \in \mathbb{C}$ such that
$$
T(p) = c_p\cdot p' + d_p \cdot p.
$$
\end{cor}
\begin{proof}
Since $T$ does not increase the degree, then we have that $n=1$ in the statement of Theorem \ref{t1}. Using the fact that Leibniz mappings vanish at $0$ and $1$ we derive the latter part.
\end{proof}

\begin{ackn}
 The research of Eszter Gselmann has been supported by project no.~K134191 that has been implemented with the support provided by the National Research, Development and Innovation Fund of Hungary,
financed under the K\_20 funding scheme. 

The research has been accomplished during the stay of Włodzimierz Fechner at the University of Debrecen, covered by the Visegrad Fellowship \#62320104.
\end{ackn}

\bibliographystyle{plain}


\end{document}